\documentclass[11pt]{amsart}
\usepackage[utf8]{inputenc}
\usepackage[T1]{fontenc}
\usepackage[english]{babel}
\usepackage{amsthm}
\usepackage{amsmath}
\usepackage{amsfonts}
\usepackage{amssymb}
\usepackage{graphicx}
\usepackage{todonotes}
\usepackage{mathtools}
\usepackage{xcolor}
\newtheorem{theorem}{\hspace*{\parindent}Theorem}
\newtheorem{lemma}{\hspace*{\parindent}Lemma}

\newtheorem{prop}{\hspace*{\parindent}Proposition}

\newcommand{\genbinom}[3]{%
  \genfrac{[}{]}{0pt}{#1}{#2}{#3}%
  }
\newcommand{\qbinom}[2]{\genbinom{}{#1}{#2}}

\def\N{\mathbb{N}}
\def\Z{\mathbb{Z}}

\def\a{\mathbf{a}}
\def\b{\mathbf{b}}

\def\m{\mathbf{m}}
\def\n{\mathbf{n}}

\newcommand{\res}{\mathop{\rm res}}

\title{A new identity for the sum of products of generalized basic hypergeometric functions}

\author{S.I.\:Kalmykov}
\address{School of mathematical sciences, Shanghai Jiao Tong University, 800 Dongchuan RD, Shanghai 200240, China}
\email{kalmykovsergei@sjtu.edu.cn}
\thanks{First author supported by NSFC grant 11901384}

\author{ D.\:Karp}
\address{Faculty of Mathematics and Statistics, Ton Duc Thang University, 19 Ngyuen Huu Tho Street, Ho Chi Minh City, Vietnam}
\email[Corresponding author]{dmitriibkarp@tdtu.edu.vn}

\author{A.\:Kuznetsov}
\address{Department of Mathematics and Statistics, York University,
Toronto, Ontario, M3J 1P3
Canada}
\email{akuznets@yorku.ca}
\thanks{Research of A.K. supported by the Natural Sciences and Engineering Research Council of Canada.}

\date{\today}
\begin{document}
\maketitle

 \begin{center}
  \emph{To the memory of Richard Askey}
\end{center}

\begin{abstract}
We prove a duality relation for generalized basic hypergeometric functions. It forms a $q$-extension of a recent result of the second and the third named authors and  generalizes both a $q$-hypergeometric identity due to the third named author (jointly with Feng and Yang) and a recent identity for the Heine's  ${}_2\phi_{1}$ function due to Suzuki. We further explore various consequences of our identity leading to several presumably new multi-term relations for both terminating and non-terminating generalized basic hypergeometric series.  Moreover, we give confluent versions of our results and furnish a number of explicit examples.
\end{abstract}

\bigskip


\bigskip


\bigskip

\section{Introduction}

Basic hypergeometric series are nearly as old as the standard ones with certain particular cases considered already by Euler and Jacobi. The general case was introduced by Heine in 1846 about 30 years after Gauss presented the ordinary hypergeometric series.  Naturally, over these 174 years of history, a huge number of identities in the form of transformation and summation formulas were discovered with deep connections to number theory, orthogonal polynomials, mathematical physics and many other fields.  The standard reference is the book by Gasper and Rahman \cite{Gr}. Another comprehensive treatment can be found in \cite{Ernst}, while more accessible introduction is in \cite{KacCheung}. A nice collection \cite{Frontiers} reflects many aspects of more modern developments.

The purpose of this paper is to present a $q$-analogue of the ordinary hypergeometric identity discovered recently in \cite{KKAMS} by the second and the third named authors.  This identity reduces a sum of products of the generalized hypergeometric functions to a Laurent polynomial (with explicit lower and upper degrees) times a power of $(1-z)$.   Particular cases of it were discovered previously by Ebisu \cite{Ebisu} (as an ingredient of his calculation of the coefficients of the contiguous relations for the Gauss hypergeometric function ${}_2F_{1}$),  by the third named author jointly with Feng and Yang in \cite{FengKuznYang} and by the first and second named authors in \cite{KalmKarp}.  Another identity of this type was deduced by Beukers and Jouhet in \cite{BJ}, but only a particular case of their formula reduces to our identity from \cite{KKAMS}, while in their full generality these results are independent.  The work by Feng, Kuznetsov and Yang also a contains a $q$-extension \cite[Theorem~2]{FengKuznYang}.  More recently, Yamaguchi \cite{Yamaguchi} extended the results of Ebisu to the $q$-case by calculating the coefficients of contiguous relations for the basic hypergeometric series ${}_2\phi_{1}$ (see definition \eqref{eq:phi} below).  A by-product of her result is an identity for a quadratic form composed of ${}_2\phi_{1}$ functions \cite[Lemma~3]{Yamaguchi}.  She further generalized this identity in \cite{Suzuki}. Our identity proved in Theorem~\ref{thm:main} below generalizes both the Feng, Kuznetsov and Yang formula obtained by setting $n_i=0$, $i=1,\ldots,r$, and the most general Suzuki's formula \cite[Lemma~15]{Suzuki} obtained by setting $r=2$, applying Heine's transformation and renaming parameters, as explained in the remark following the proof of Theorem~\ref{thm:main}.  Our main result leads to identities for finite sums of the terminating and non-terminating generalized basic hypergeometric functions.  We further present a confluent form of our formula in Theorem~\ref{thm:confluent}.  Finally, let us remark that another apparently unrelated identity for a quadratic form of the basic hypergeometric series was established in 2015 in \cite[Theorem~1.1]{GITZ}.

\section{Main results}
Before we turn to our results let us introduce some notation and definitions.
We will use the definition of the $q$-shifted factorial valid for both positive and negative values of the index (\cite[(1.2.15) and (1.2.28)]{Gr}):
\begin{equation}\label{eq:q-shifted}
(a;q)_0=1,~~~~(a;q)_n=\prod\limits_{k=0}^{n-1}(1-aq^k), \ \ (a;q)_{-n} = \prod_{k=1}^{n} \dfrac{1}{(1-a/q^k)},~~n\in\N.
\end{equation}
It is easy to  see that
\begin{align}\label{eq:qPochhammerId}
(q^a;q)_n&=\frac{(-1)^nq^{an+n(n-1)/2}}{(q^{1-a};q)_{-n}}
\\ \nonumber
&=(-1)^nq^{an+n(n-1)/2}(q^{-a};1/q)_{n}=(-1)^n(q^{1-a-n};q)_{n}q^{an+n(n-1)/2}.
\end{align}
This definition works for any complex $a$ and $q$. If we restrict $q$ to $|q|<1$, we can also define
$$
(a;q)_{\infty}=\lim\limits_{n\to\infty}(a;q)_n,
$$
where the limit can be shown to exist as a finite number for all complex $a$.  The $q$-gamma function is given by  \cite[(1.10.1)]{Gr}, \cite[(21.16)]{KacCheung}
\begin{equation}\label{eq:q-Gamma-def}
\Gamma_q(z)=(1-q)^{1-z}\frac{(q;q)_{\infty}}{(q^z;q)_{\infty}}
\end{equation}
for $|q|<1$ and all complex $z$ such that $q^{z+k}\ne1$ for all $k\in\N_0$.  Comparing this definition with \eqref{eq:q-shifted} we immediately  see that
\begin{equation}\label{eq:q-Gamma-q-Poch}
\frac{\Gamma_q(z+k)}{\Gamma_q(z)}=\frac{(q^z;q)_{k}}{(1-q)^k}
\end{equation}
for both positive and negative $k$.

Suppose $r,s$ are positive integers, $\a =(a_1,\dots, a_r)\in\mathbb{C}^r$, $\b=(b_1,\dots, b_{s})\in \mathbb{C}^{s}$.
The symbol $\a_{[j]}$ (where $1\le j \le r$) stands for the vector $\a$ with $j$-th component omitted:
$$
\a_{[j]}=(a_1,\dots,a_{j-1},a_{j+1},\dots,a_r),
$$
and, by definition,
$$
\a+\beta= (a_1+\beta,\dots, a_r+\beta).
$$
For $\n=(n_1,\dots, n_r)\in \mathbb{Z}^r$ we use the following abbreviations for the products:
$$
(\a)_{\n} = (a_1;q)_{n_1}(a_2;q)_{n_2}\cdots (a_r;q)_{n_r}~\text{and}~\Gamma_q(\a)=\Gamma_q(a_1)\cdots\Gamma_q(a_r).
$$
We will further write $q^{\a}=(q^{a_1},\dots,q^{a_r})$.  The basic hypergeometric function is defined as follows   \cite[formula (1.2.22)]{Gr}
\begin{equation} \label{eq:phi}
{_r\phi_{s}}\left(\!\begin{array}{c} \a \\ \b   \end{array}\!\vline\,q,z\right)
=\sum_{n=0}^{\infty}\frac{(a_1;q)_n(a_2;q)_n \cdots (a_r;q)_n}{(b_1;q)_n (b_2;q)_n \cdots (b_s;q)_n(q;q)_n} \left[(-1)^n q^{\binom{n}{2}}\right]^{1+s-r}z^n,
\end{equation}
where $r\le{s+1}$, and the series converges for all $z$ if $r\le{s}$ and for $|z|<1$ if $r=s+1$ \cite[section~1.2]{Gr}.
%In particular,
%\begin{equation}\label{eq:r-phi-r-1}
%{_r\phi_{r-1}}\left(\!\begin{array}{l}\a\\\b\end{array}\!\%vline\,q,z\right)
%={_{r}\phi_{r-1}}(\a;\b;q,z)=\sum_{n=0}^{\infty}\frac{(\a;%q)_{n}}{(\b;q)_{n}(q;q)_n}z^n,
%\end{equation}
%where $\a=(a_1,\ldots, a_{r})$, $\b=(b_1,\ldots,b_{r-1})$.

We will also need the version  of the  basic hypergeometric function considered by Bailey \cite[section~8.1]{Bailey} and Slater \cite[(3.2.1.11)]{Slater} which is defined as follows
\begin{equation} \label{eq:phitilde}
{_r\hat{\phi}_{s}}\left(\!\begin{array}{c} \a  \\
\b  \end{array}\!\vline\,q,z\right)
=\sum_{n=0}^{\infty}\frac{(a_1;q)_n (a_2;q)_n \cdots (a_r;q)_n}{(b_1;q)_n (b_2;q)_n \cdots (b_s;q)_n(q;q)_n} z^n.
\end{equation}
Observe  \cite[p.5]{Gr} that the series \eqref{eq:phi} has the property that if we replace $z$ by $z/b_r$ and let $b_r\to\infty$, then the resulting series is again of the form \eqref{eq:phi} with $r$ replaced by $r-1$ (this is known as \emph{the confluence process}). At the same time the Bailey and Slater series  \eqref{eq:phitilde} can be obtained from the $r=s+1$ case of \eqref{eq:phi} by setting some of the parameters equal to zero:
\begin{equation} \label{eq:phitildephi}
{_r\hat{\phi}_{s}}\left(\!\begin{array}{c}a_1, \ldots, a_r \\  b_1, \ldots, b_s \end{array}\!\vline\,q,z\right)
={_{s+1}\phi_{s}}\left(\!\begin{array}{c} a_1,\ldots,a_r,0,\ldots,0\\b_1, \ldots, b_s  \end{array}\!\vline\,q,z\right).
\end{equation}
Note also that in the case $r=s+1$ the functions ${}_r\phi_{s}$ and  ${}_r\hat{\phi}_{s}$ coincide. Since the base $q$ is the same in all formulas in this paper, in what follows we omit it from the notation of the $q$-hypergeometric functions.
Denote by $\qbinom{n}{j}_q$ the standard $q$-binomial coefficient
defined by
$$
\qbinom{n}{j}_q=\frac{(q;q)_{n}}{(q;q)_{j}(q;q)_{n-j}}=\frac{(q^{n-j+1};q)_{j}}{(q;q)_{j}},
$$
see, for instance, \cite[Chapter~7]{KacCheung} or \cite[p.24]{Gr}. Put as customary $(x)_{+}=\max(0,x)$ for any real $x$. Our main result is the following

\begin{theorem}\label{thm:main}
Suppose that $r\ge2$ is an integer, $q$ with $|q|<1$ a complex number,  the vector $\a\in\mathbb{C}^r$ satisfies $a_i-a_j\notin\Z$ for $1\le i < j\le r$ and the vector $\b\in\mathbb{C}^{r}$ is arbitrary; assume further that $\m,\n\in\Z^r$, $t\in\Z$. Then, for the numbers $\beta_k$, given by
\begin{multline}\label{eq:beta-final}
\beta_{k}=\sum\limits_{j=\max(-n_{\max},k-p-(t)_{+}-1)}^{k}(-1)^{k-j}
\sum_{g+h=k-j}\qbinom{p+1}{h}_q\qbinom{(t)_{+}}{g}_q q^{[h(h-1)+g(g-1)]/2}W^h
\\
\times\sum_{i=1}^{r}\frac{q^{a_i(1-t)}(q^{1-\b+a_{i}};q)_{\m+k}}{(q;q)_{k+n_i}(q^{a_{i}-\a_{[i]}};q)_{\n_{[i]}+k+1}}
{_{2r}\phi_{2r-1}}\left(\!\begin{array}{l}q^{-k-n_i},q^{\b-a_i},q^{\a_{[i]}-a_{i}-\n_{[i]}-k}\\q^{\b-a_i-\m-k},q^{1-a_i+\a_{[i]}}\end{array}\!\vline\, q^{N-M+r-1+t}\right),
\end{multline}
where  $W=q^{t+r-1}\prod_{i=1}^r q^{a_i-b_i}$, the following identity holds
\begin{multline}\label{eq:MainIdentityNew}
\sum_{i=1}^{r} q^{a_i(1-t)}\frac{(q^{1-\b+a_i};q)_{\m-n_i}z^{-n_i}}{(q^{a_i-\a_{[i]}};q)_{\n_{[i]}-n_i+1}}
{_r\phi_{r-1}}\left(\!\begin{array}{l}q^{\b-a_i}\\q^{1+\a_{[i]}-a_i}\end{array}\!\vline\, Wz\right)
{_r\phi_{r-1}}\left(\!\begin{array}{l}q^{1-\b+a_i+\m-n_i}\\q^{1-\a_{[i]}+a_i+\n_{[i]}-n_i}\end{array}\!\vline\,z\right)
\\
=\frac{1}{(Wz;q)_{p+1}(z;q)_{(t)_{+}}}\sum\limits_{k=-n_{max}}^{p+(t)_{+}-m_{min}}\beta_kz^k.
\end{multline}
Here $m_{\min} = \min\limits_{1\le i \le r} (m_i)$, $n_{\max}=\max\limits_{1\le i \le r} (n_i)$, $M=\sum\limits_{i=1}^r m_i$,  $N=\sum\limits_{i=1}^r n_i$ and $p=\max\{-1, M-N-r-t+1\}$.
\end{theorem}

\begin{proof}
Using the Cauchy product we  expand
\begin{multline}\label{eq:Cauchy_expnd1}
S=\sum_{i=1}^{r}q^{a_i(1-t)}\frac{(q^{1-\b+a_i};q)_{\m-n_j}z^{-n_i}}{(q^{a_i-\a_{[i]}};q)_{\n_{[i]}-n_i+1}}
{_r\phi_{r-1}}\left(\!\begin{array}{l}q^{\b-a_i}\\q^{1+\a_{[i]}-a_i}\end{array}\!\vline\, Wz\right)
{_r\phi_{r-1}}\left(\!\begin{array}{l}q^{1-\b+a_i+\m-n_i}\\q^{1-\a_{[i]}+a_i+\n_{[i]}-n_i}\end{array}\!\vline\,z\right)
\\
=\sum_{i=1}^{r}\frac{(q^{1-\b+a_i};q)_{\m-n_i}z^{-n_i}}{(q^{a_i-\a_{[i]}};q)_{\n_{[i]}-n_i+1}}
\sum_{k=0}^{\infty}z^k\sum_{j=0}^{k}\frac{ q^{a_i(1-t)}(q^{\b-a_i};q)_{j}(q^{1-\b+a_i+\m-n_i};q)_{k-j}W^{j}}
{(q^{1+\a_{[i]}-a_i};q)_{j}(q^{1-\a_{[i]}+a_i+\n_{[i]}-n_i};q)_{k-j}(q;q)_j(q;q)_{k-j}}
\\
=\sum_{i=1}^{r}\sum_{k=0}^{\infty}z^{k-n_{i}}
\sum_{j=0}^{k}\frac{q^{a_i(1-t)}(q^{1-\b+a_i};q)_{\m-n_i}(q^{\b-a_i};q)_{j}(q^{1-\b+a_i+\m-n_i};q)_{k-j}W^{j}}
{(q^{a_i-\a_{[i]}};q)_{\n_{[i]}-n_i+1}(q^{1+\a_{[i]}-a_i};q)_{j}(q^{1-\a_{[i]}+a_i+\n_{[i]}-n_i};q)_{k-j}(q;q)_{j}(q;q)_{k-j}}
\\
=\sum_{i=1}^{r}\sum_{k=0}^{\infty}z^{k-n_i}\sum_{j=0}^{k}\gamma_{i,j}^k =\sum_{i=1}^{r}\sum_{k_i=-n_i}^{\infty}z^{k_i}\sum_{j=0}^{k_i+n_i} \gamma_{i,j}^{k_i+n_i},
\end{multline}
where we denoted
\begin{equation}\label{eq:gamma_ij}
\gamma_{i,j}^k=\frac{q^{a_i(1-t)}(q^{1-\b+a_i};q)_{\m-n_i}(q^{\b-a_i};q)_{j}(q^{1-\b+a_i+\m-n_i};q)_{k-j}W^{j}}
{(q^{a_i-\a_{[i]}};q)_{\n_{[i]}-n_i+1}(q^{1+\a_{[i]}-a_i};q)_{j}(q^{1-\a_{[i]}+a_i+\n_{[i]}-n_i};q)_{k-j}(q;q)_{j}(q;q)_{k-j}}.
\end{equation}
Using \eqref{eq:qPochhammerId} and \eqref{eq:q-Gamma-q-Poch} we obtain  for each term in the numerator on the right-hand side of 
\eqref{eq:gamma_ij}
\begin{align*}
  & (q^{1-b_l+a_i};q)_{m_l-n_i}(q^{b_l-a_i};q)_{j}(q^{1-b_l+a_i+m_l-n_i};q)_{k-j}
    \\
    &=(q^{b_l-a_i};q)_{j}(1-q)^{m_l-n_i+k-j}\dfrac{\Gamma_q(1-b_l+a_i+m_l-n_i)}{\Gamma_q(1-b_l+a_i)}  \dfrac{\Gamma_q(1-b_l+a_i+m_l-n_i+k-j)}{\Gamma_q(1-b_l+a_i+m_l-n_i)}
    \\
    &=(q^{b_l-a_i};q)_{j}(q^{1-b_l+a_i};q)_{m_l-n_i+k-j}
    \\
    &=(-1)^{j}q^{(b_l-a_i)j+j(j-1)/2}(q^{1-b_l+a_i-j};q)_{j}(q^{1-b_l+a_i};q)_{m_l-n_i+k-j}
    \\
    &=(-1)^{j}q^{(b_l-a_i)j+j(j-1)/2}(q^{1-b_l+a_i-j};q)_{m_l-n_i+k}.
\end{align*}
After a similar calculation for the denominator terms in \eqref{eq:gamma_ij}, we arrive at
\begin{equation}\label{eq:gamma_compact}
\gamma_{i,j}^k=\frac{(-1)^{j}q^{j(j-1)/2}q^{a_i(1-t)+tj}(q^{1-\b+a_{i}-j};q)_{\m+k-n_{i}}}{(q^{a_{i}-\a_{[i]}-j};q)_{\n_{[i]}+k-n_{i}+1}(q;q)_{j}(q;q)_{k-j}},
\end{equation}
which is equivalent to
\begin{equation}\label{eq:gamma_compact1}
\gamma_{i,j}^{k+n_{i}}=\frac{(-1)^{j}q^{j(j-1)/2}q^{a_i(1-t)+tj}(q^{1-\b+a_{i}-j};q)_{\m+k}}{(q^{a_{i}-\a_{[i]}-j};q)_{\n_{[i]}+k+1}(q;q)_{j}(q;q)_{k+n_{i}-j}}.
\end{equation}
According to \eqref{eq:qPochhammerId}, we have $(q;q)_{j}=\infty$ for $j<0$ so that also $(q;q)_{k-j}=\infty$ for $j>k$ and thus definition \eqref{eq:gamma_ij} implies that $\gamma_{i,j}^{k+n_i}=0$ for $j<0$ and $j>k+n_i$. These relations extend the definition of $\gamma_{i,j}^{k+n_i}$ to arbitrary $j$, and, in view of this convention, we have by rearranging \eqref{eq:Cauchy_expnd1}:
$$
S=\sum_{k=-n_{\max}}^{\infty}z^k\sum_{i=1}^{r}\sum_{j=0}^{k+n_i}\gamma_{i,j}^{k+n_i}.
$$
Next, note that $k+m_i\ge0$ for all $i=1,\dots,r$ if $k\ge-m_{\min}$. Hence, if $-m_{\min}\le -n_{\max}$, then $k\ge -m_{\min}$ for all terms in the above sum. Otherwise, if $-m_{\min}>-n_{\max}$ we can write
$$
S(z)=\sum_{k=-n_{\max}}^{-m_{\min}-1}\alpha_k z^k+\sum_{k=-m_{\min}}^{\infty}z^k\sum_{i=1}^{r}\sum_{j=0}^{k+n_i} \gamma_{i,j}^{k+n_i}
=\sum_{k=-n_{\max}}^{-m_{\min}-1}\alpha_kz^k +S_1(z),
$$
where
$$
\alpha_k=\sum_{i=1}^{r}\sum_{j=0}^{k+n_i}\gamma_{i,j}^{k+n_i}.
$$

Next, for each $k\in\Z$ define the functions
$$
f_k(z)=- \frac{(zq^{1-\b};q)_{k+\m}}{(zq^{-\a};q)_{k+\n+1}}~~~\text{and}~~~F_k(z)=z^{-t} f_k(z).
$$
According to \eqref{eq:q-shifted}, these functions are rational for all $k\in\Z$.  Further, if  $k\ge-m_{\min}$ the expression
in the numerator is a polynomial and all poles of $f_k$ come from the zeros of the denominator.  If $k+n_i+1>0$ for all $i=1,\ldots,r$, then the poles of $f_k(z)$ are at the points:
$$
(zq^{-\a};q)_{k+\n+1}=0~~\Leftrightarrow~~z=q^{a_{i}-j},~i=1,\ldots,r~\text{and}~j=0,\ldots,k+n_{i}.
$$
Our assumption $a_i-a_j\notin\Z$ for $1\le i < j\le r$ guarantees that all poles of $f_k$ are simple. The poles of $F_k$ include all poles of $f_k$ and a (possibly multiple) pole at $z=0$. Thus, after simple calculation, we obtain
$$
\res\limits_{z=q^{a_{i}-j}}F_k(z)
=\frac{(-1)^{j}q^{j(j-1)/2+a_i(1-t)+tj}q^{a_i}(q^{1-\b+a_{i}-j};q)_{\m+k}}
{(q^{a_{i}-\a_{[i]}-j};q)_{\n_{[i]}+k+1}(q;q)_{j}(q;q)_{k+n_{i}-j}}
=\gamma_{i,j}^{k+n_{i}}.
$$
If $k+n_i+1\le0$ for some values of $k$ and $i$, then $(zq^{-\a};q)_{k+n_i+1}$ has no zeros at the points $z=q^{a_{i}-j}$, $j=k+n_i,\ldots,0$, so that these points do not contribute to the sum of residues of $f_k(z)$.  In view of the above comment on the values of $\gamma_{i,j}^{k+n_{i}}$ for $j<0$, this implies that for  $k\ge-m_{\min}$
$$
\sum\limits_{\text{over all nonzero poles of~}F_k(z)}\res{F_k(z)}=\sum_{i=1}^{r}\sum_{j=0}^{k+n_i}\gamma_{i,j}^{k+n_i},
$$
where only the terms with $k+n_i\ge0$ are non-vanishing. Next, we use the fact that the sum of residues of a rational function at all finite points equals its residue at infinity \cite[(4.1.14)]{AblFok}, which is the coefficient at $z^{-1}$
in the asymptotic expansion
$$
F_{k}(z)=\sum\limits_{j=-\infty}^{M-N-r-t}C_{j}(k)z^{j}.
$$
Hence,
$$
\sum_{i=1}^{r}\sum_{j=0}^{k+n_i}\gamma_{i,j}^{k+n_i}=C_{-1}(k)- \res\limits_{z=0}{F_k(z)}.
$$
Our key observation is the following lemma.
\begin{lemma}\label{lemma1}${}$
\begin{itemize}
\item[(i)]
For $z$ large enough we have
\begin{equation}\label{fk_asymptotics}
f_k(z)=A^k z^{M-N-r} \sum\limits_{l\ge 0}  Q_l(q^{-k}) z^{-l},
\end{equation}
where
$$
A=\prod\limits_{i=1}^r q^{a_i - b_i +m_i-n_i}
$$
and for each $l\ge 0$ the function $Q_l(w)$ is a polynomial of degree $l$ whose coefficients do not depend on $k$.
\item[(ii)]
For $z$ small enough we have
\begin{equation}\label{fk_asymptotics_zero}
f_k(z)=\sum\limits_{l\ge 0}  \tilde Q_l(q^{k}) z^{l},
\end{equation}
where for each $l\ge 0$ the function $\tilde Q_l(w)$ is a polynomial of degree $l$ whose coefficients do not depend on $k$.
\end{itemize}
\end{lemma}
\begin{proof}
Consider first the function
$$
g_k(z;\alpha,\beta)=\frac{(\beta z;q)_k}{(\alpha z;q)_k}.
$$
We compute
$$
g_k(z;\alpha,\beta)=\prod\limits_{j=0}^{k-1}
\frac{1-\beta z q^j}{1-\alpha z q^j}=(\beta/\alpha)^k \prod\limits_{j=0}^{k-1}
\frac{1-\frac{1}{\beta z} q^{-j}}{1-\frac{1}{\alpha z} q^{-j}}
$$
and
\begin{align*}
\ln(g_k(z;\alpha,\beta))&=
k \ln(\beta/\alpha) + \sum\limits_{j=0}^{k-1} \Big( \ln(1-\beta^{-1} z^{-1} q^{-j}) - \ln(1-\alpha^{-1} z^{-1} q^{-j}) \Big)\\
&=
k \ln(\beta/\alpha) - \sum\limits_{j=0}^{k-1} \sum\limits_{l\ge 0}
\frac{1}{l z^l q^{jl}} \Big( \beta^{-l}-\alpha^{-l} \Big)
\\&=k \ln(\beta/\alpha) - \sum\limits_{l\ge 0} \frac{1}{l z^l}
\times \Big( \beta^{-l}-\alpha^{-l} \Big)
\sum\limits_{j=0}^{k-1} q^{-jl}\\
&=
k \ln(\beta/\alpha) - \sum\limits_{l\ge 0} \frac{1}{l z^l}
\times \Big( \beta^{-l}-\alpha^{-l} \Big)\frac{1-q^{-kl}}{1-q^{-l}}.
\end{align*}
Here in the second step we used Taylor series for $\ln(1-w)$. The above series converges for all $z$ large enough.
We can express this result in the following form: for $z$ large enough
$$
\ln(g_k(z;\alpha,\beta))=k \ln(\beta/\alpha) + \sum\limits_{l\ge 0} z^{-l} P_l(q^{-k})
$$
for some polynomials $P_l$ of degree $l$. After exponentiating both sides of the above identity
and collecting the powers of $z^{-l}$ on the right-hand side we obtain (for all $z$ large enough)
\begin{equation}\label{asympt_gk}
g_k(z;\alpha,\beta)=(\beta/\alpha)^k \sum\limits_{l \ge 0} z^{-l} \tilde P_l(q^{-k})
\end{equation}
where every function $\tilde P_l$ is a polynomial of degree $l$.

Next we use the identity $(w;q)_{k+l}=(w;q)_l (wq^l;q)_k$ and express $f_k$ in the following form
\begin{equation}\label{eq:f_k_R}
f_k(z)=-\frac{(zq^{1-{\mathbf b}};q)_{k+{\mathbf m}}}{(zq^{-{\mathbf a}};q)_{k+{\mathbf n}+1}}
=R(z) \prod\limits_{i=1}^r g_k(z;q^{-a_i+n_i+1},q^{-b_i+m_i+1}),
\end{equation}
where
$$
R(z)=-\frac{(zq^{1-{\mathbf b}};q)_{{\mathbf m}}}{(zq^{-{\mathbf a}};q)_{{\mathbf n}+1}}.
$$
The function $R$ has series expansion (for all $z$ large enough)
\begin{equation}\label{asympt_R}
R(z)=z^{M-N-r} \sum\limits_{l\ge 0} c_l z^{-l},
\end{equation}
for some coefficients $c_l$.
Using \eqref{asympt_gk} we find the expansion
\begin{equation}\label{asympt_product_gk}
\prod\limits_{i=1}^r g_k(z;q^{-a_i+n_i+1},q^{-b_i+m_i+1})=
A^k \sum\limits_{l \ge 0} z^{-l} \tilde Q_l(q^{-k}),
\end{equation}
for some polynomials $\tilde Q_l$ of degree $l$. Combining \eqref{asympt_R} and \eqref{asympt_product_gk} we arrive at the desired result \eqref{fk_asymptotics}.

Now we derive expansion of $g_k(z;\alpha,\beta)$ at $z=0$.
We compute
\begin{align*}
\ln(g_k(z;\alpha,\beta))&=
 \sum\limits_{j=0}^{k-1} \Big( \ln(1-\beta z q^{j}) - \ln(1-\alpha z q^j) \Big)\\
&=
 - \sum\limits_{j=0}^{k-1} \sum\limits_{l\ge 0}
\frac{1}{l} z^l q^{jl} \Big( \beta^{l}-\alpha^{l} \Big)
\\&=- \sum\limits_{l\ge 0} \frac{1}{l} z^l
\times \Big( \beta^{l}-\alpha^{l} \Big)
\frac{1-q^{kl}}{1-q^{l}}.
\end{align*}
The above series converges for $z$ small enough.
We can express this result in the following form: for $z$ small enough
$$
\ln(g_k(z;\alpha,\beta))=\sum\limits_{l\ge 0} z^{l} \hat P_l(q^{k})
$$
for some polynomials $\hat P_l$ of degree $l$. After exponentiating both sides of the above identity
and collecting the powers of $z^{l}$ on the right-hand side we obtain (for all $z$ small enough)
\begin{equation*}\label{asympt_gkzero}
g_k(z;\alpha,\beta)= \sum\limits_{l \ge 0} z^{l} \check P_l(q^{k})
\end{equation*}
where every function $\check P_l$ is a polynomial of degree $l$.
Using formula \eqref{eq:f_k_R} and the fact that the rational function $R$ has no poles at zero we obtain the desired expansion \eqref{fk_asymptotics_zero}.
\end{proof}

The above Lemma implies that the coefficient $C_{-1}(k)$ at $z^{-1}$ equals $0$ if $M-N-r-t<-1$;  and equals $A^kQ_p(q^{-k})$ if $M-N-r-t=p-1$, $p\in\N_0$, where $Q_p(y)$ is a polynomial of degree $p$ in $y$.
Similarly, the residue of $F_k(z)$ at $z=0$ is equal to zero if $t\le 0$ and it is equal to $\tilde{Q}_{t-1}(q^k)$ if $t\ge 1$, where $\tilde{Q}_{t-1}(y)$ is a polynomial of degree $t-1$ in $y$.

 Hence, we will have, assuming $Q_{-1}(y)\equiv0$ and  $\tilde{Q}_{-n}(y)\equiv0$ for all $n\in\N$,
\begin{align*}
S(z)&=\sum_{k=-n_{\max}}^{-m_{\min}-1}\alpha_kz^k+\sum_{k=-m_{\min}}^{\infty}z^k\sum_{i=1}^{r}\sum_{j=0}^{k+n_i}\gamma_{i,j}^{k+n_i}
\\
&=\sum_{k=-n_{\max}}^{-m_{\min}-1}\alpha_kz^k+\sum_{k=-m_{\min}}^{\infty}z^k[A^k Q_p(q^{-k})+\tilde Q_{t-1}(q^k)]
\\
&=\sum_{k=-n_{\max}}^{-m_{\min}-1}\alpha_kz^k+\sum_{k=-m_{\min}}^{\infty}(Az)^k(\alpha_{p,0}+\alpha_{p,1}q^{-k}+\cdots+\alpha_{p,p}q^{-pk})\\
& \qquad \qquad\qquad + \sum_{k=-m_{\min}}^{\infty}z^k(\beta_{t-1,0}+\beta_{t-1,1}q^{k}+\cdots+\beta_{t-1,t-1}q^{k(t-1)})
\\
&=\sum_{k=-n_{\max}}^{-m_{\min}-1}\alpha_kz^k+(z)^{-m_{min}}\Bigg(\frac{\tilde \alpha_{p,0}}{1-Az}+\frac{\tilde \alpha_{p,1}q^{m_{min}}}{1-Az/q}+\cdots+\frac{\tilde \alpha_{p,p}q^{pm_{min}}}{1-Az/q^p}\\
&\qquad \qquad +
\frac{\tilde \beta_{t-1,0}}{1-z}+\frac{\tilde \beta_{t-1,1}}{1-qz}+\cdots+\frac{\tilde \beta_{t-1,t-1}}{1- q^{t-1}z}\Bigg).
\end{align*}
Noting that the common denominator on the RHS equals $(Wz,q)_{p+1}(z;q)_{(t)_{+}}$ and we arrive at \eqref{eq:MainIdentityNew}.

Finally we compute the numbers $\beta_k$.  We will need the following two easily verifiable identities
\begin{equation}\label{eq:qpochhammertricks}
(q;q)_{k-j}=\frac{(-1)^jq^{j(j-1)/2-kj}(q;q)_{k}}{(q^{-k};q)_{j}},
~~~(q^{\alpha-j};q)_{m}=\frac{(q^{\alpha};q)_{m}(q^{1-\alpha};q)_{j}}{q^{mj}(q^{1-\alpha-m};q)_{j}}.
\end{equation}
Application of these identities to \eqref{eq:gamma_compact1} after some simplifications leads to
\begin{equation*}\label{eq:gamma_pure_j}
\gamma_{i,j}^{k+n_{i}}=\frac{q^{a_i(1-t)}q^{(N-M+r-1+t)j}(q^{1-\b+a_{i}};q)_{\m+k}(q^{-k-n_i};q)_{j}(q^{\b-a_{i}};q)_{j}(q^{\a_{[i]}-a_{i}-\n_{[i]}-k};q)_{j}}
{(q;q)_{k+n_i}(q^{a_{i}-\a_{[i]}};q)_{\n_{[i]}+k+1}(q^{\b-a_i-\m-k};q)_{j}(q^{1-a_i+\a_{[i]}};q)_{j}(q;q)_{j}},
\end{equation*}
so that
\begin{multline*}
\alpha_k=\sum_{i=1}^{r}\sum_{j=0}^{k+n_i}\gamma_{i,j}^{k+n_i}
\\
=\sum_{i=1}^{r}\frac{q^{a_i(1-t)}(q^{1-\b+a_{i}};q)_{\m+k}}{(q;q)_{k+n_i}(q^{a_{i}-\a_{[i]}};q)_{\n_{[i]}+k+1}}
{_{2r}\phi_{2r-1}}\left(\!\begin{array}{l}q^{-k-n_i},q^{\b-a_i},q^{\a_{[i]}-a_{i}-\n_{[i]}-k}\\q^{\b-a_i-\m-k},q^{1-a_i+\a_{[i]}}\end{array}\!\vline\, q^{N-M+r-1+t}\right),
\end{multline*}
where each term with $k+n_i<0$ is assumed to vanish.  Next, we recall  the Gauss expansion
$$
(a;q)_n=\sum\limits_{j=0}^{n}\qbinom{n}{j}_q q^{j(j-1)/2}(-a)^j,
$$
where $\qbinom{n}{j}_q$ is the $q$-binomial coefficients.  This yields
$$
(Wz;q)_{p+1}(z;q)_{(t)_{+}}=\sum_{j=0}^{p+1+(t)_{+}}(-z)^j\sum_{i+l=j}\qbinom{p+1}{l}_q\qbinom{(t)_{+}}{i}_q q^{[l(l-1)+i(i-1)]/2}W^l=\sum_{j=0}^{p+1+(t)_{+}}D_{j}(-z)^j,
$$
where the last equality is the definition of $D_{j}$.  Note that
$$
D_j=\qbinom{p+1}{j}_qq^{j(j-1)/2}W^j
$$
if $t\le0$.

Multiplying both sides of \eqref{eq:MainIdentityNew} by $(Wz;q)_{p+1}(z;q)_{(t)_{+}}$ and applying the above expansion, we obtain:
$$
\sum_{j=0}^{p+1+(t)_{+}}D_{j}(-z)^j\sum\limits_{k=-n_{\max}}^{\infty}\alpha_{k}z^k=\sum\limits_{k=-n_{\max}}^{p+(t)_{+}-m_{\min}}\beta_{k}z^{k}.
$$
Multiplying both sides by $z^{n_{\max}}$, changing $k+n_{\max}\to{k}$ and writing $\hat{\beta}_k=\beta_{k-n_{\max}}$,
$\hat{\alpha}_k=\alpha_{k-n_{\max}}$, we get
$$
\sum_{j=0}^{p+1+(t)_{+}}D_{j}(-z)^j
\sum\limits_{k=0}^{\infty}\hat{\alpha}_{k}z^k
=\sum\limits_{s=0}^{\infty}z^{s}\sum\limits_{j+k=s}D_{j}(-1)^j\hat{\alpha}_{k}
=\sum\limits_{k=0}^{p+(t)_{+}-m_{\min}+n_{\max}}\hat{\beta}_{k}z^{k}.
$$
In view of $D_{j}=0$ for $j>p+(t)_{+}+1$ this implies that
$$
\beta_{s-n_{\max}}=\hat{\beta}_s=\sum\limits_{j=0}^{\min(s,p+(t)_{+}+1)}D_{j}(-1)^j\hat{\alpha}_{s-j}
=\sum\limits_{j=0}^{\min(s,p+(t)_{+}+1)}D_{j}(-1)^j\alpha_{s-j-n_{\max}}
$$
for $s=0,\ldots,p+(t)_{+}-m_{\min}+n_{\max}$. Returning to $k=s-n_{\max}$ we obtain by changing the index of summation $j\to{k-j}$:
$$
\beta_{k}=\sum\limits_{j=0}^{\min(k+n_{\max},p+(t)_{+}+1)}(-1)^jD_{j}\alpha_{k-j}
=\sum\limits_{j=\max(-n_{\max},k-p-(t)_{+}-1)}^{k}(-1)^{k-j}D_{k-j}\alpha_{j}
$$
for $k=-n_{\max},\ldots,p+(t)_{+}-m_{\min}$. Substituting the formula for $\alpha_{j}$ we finally arrive at  \eqref{eq:beta-final}.
\end{proof}

\textbf{Remark.} In a recent paper \cite{Suzuki} Suzuki studied the coefficients of the contiguous relation for ${}_2\phi_{1}$ series.  As an important tool she established  an identity for ${}_2\phi_{1}$ basic hypergeometric function \cite[Lemma~15]{Suzuki}.
After renaming parameters and rearranging  terms her formula can be written as follows:
\begin{align}\label{eq:Suzuki}
&\sum_{i=1}^{2} c_i\frac{(q^{1-b_1+a_i};q)_{m_1-n_i}(q^{1-b_2+a_i};q)_{m_2-n_i}z^{-n_i}}
{(q^{a_i-a_{i'}};q)_{n_{i'}-n_i+1}}{_2\phi_{1}}\left(\!\begin{array}{l}q^{b_1-a_i},q^{b_2-a_i}\\q^{1+a_{i'}-a_i}\end{array}\!\vline\, z\right)
\\ \nonumber
&\qquad \times
{_2\phi_{1}}\left(\!\begin{array}{l}q^{1-b_1+a_i+m_1-n_i},q^{1-b_2+a_i+m_2-n_i}\\q^{1-a_{i'}+a_i+n_{i'}-n_i}\end{array}\!\vline\,Wz\right)
=\frac{z^{-\max(n_1,n_2)}P_l(z)}{(z;q)_{(t)_{+}}},
\end{align}
where $i'=3-i$, $a_1-a_2\notin\Z$ and $m_1,m_2,n_1,n_2$  are integers satisfying $m_1+m_2\leq n_1+n_2$, $P_l$ is a polynomial of a fixed degree $l$ expressible in terms of  $m_i$, $n_i$ and $t$,  and
$$
W=q^{\sum_{k=1}^{2}(b_k-a_k+n_k-m_k)+t-1},~~~c_i=q^{a_i(1-t)}q^{\sum_{k=1}^{2}(b_k-a_i)(m_k-n_i)}q^{-\sum_{k=1}^{2}(m_k-n_i)(m_k-n_i+1)/2}.
$$
After applying Heine's transformation formula
$$
{_2\phi_{1}}\left(\!\begin{array}{l}q^{a},q^{b}\\q^{c}\end{array}\!\vline\,z\right)=
\frac{(q^{a+b-c}z;q)_{\infty}}{(z;q)_{\infty}}
{_2\phi_{1}}\left(\!\begin{array}{l}q^{c-a},q^{c-b}\\q^{c}\end{array}\!\vline\,q^{a+b-c}z\right)
$$
to each
${_2 \phi_{1}}$ function in the left-hand side of \eqref{eq:Suzuki}, renaming the parameters according to
$a_j \mapsto -a_j$, $b_j \mapsto 1-b_j$, $m_j \mapsto -m_j$ and $n_j \mapsto -n_j$ and simplifying the result, formula \eqref{eq:Suzuki} reduces  to the $r=2$ case of \eqref{eq:MainIdentityNew}.
Note also that in her previous work \cite{Yamaguchi} the author discovered an identity for ${}_2\phi_{1}$ given in \cite[Lemma~3]{Yamaguchi}.  This identity is a particular case of \cite[Lemma~15]{Suzuki} and hence, also reduces to $r=2$ case of our formula with $t=0$.

In the following two statements we will explore some consequences of Theorem~\ref{thm:main}.  Namely, we will derive identities for finite sums of terminating and non-terminating ${_r\phi_{r-1}}$.

\begin{prop}\label{prop:sum-beta1}
Assume that all conditions of Theorem \ref{thm:main} are satisfied and that $|W|<1$. Then
\begin{equation}\label{eq:sum-beta1}
\sum_{i=1}^{r} q^{a_i(1-t)}\frac{(q^{1-\b+a_i};q)_{\infty}}{(q^{a_i-\a_{[i]}};q)_{\infty}}
{_r\phi_{r-1}}\left(\!\begin{array}{l}q^{\b-a_i}\\q^{1+\a_{[i]}-a_i}\end{array}\!\vline\, W\right)=
\frac{(q^t;q)_{\infty}}{(W;q)_{p+1}}\sum\limits_{k=-n_{max}}^{p+(t)_{+}-m_{min}}\beta_k.
\end{equation}
Note that $(q^t;q)_{\infty}=0$ for $t\le 0$ , so that the right-hand side vanishes for such values of $t$.
\end{prop}
\begin{proof}
The proof is based on the following asymptotic result:
\begin{equation}\label{eq:r_phi_r-1_asymptotics}
(1-z){_r\phi_{r-1}}\left(\!\begin{array}{c} {\mathbf v} \\ {\mathbf w}  \end{array}\!\vline\,q,z\right)
\to
\frac{({\mathbf v};q)_{\infty}}{({\mathbf w};q)_{\infty}(q;q)_{\infty}}, \;\;\; z\to 1-.
\end{equation}
This result is probably known, but we could not locate it in the literature. The proof of \eqref{eq:r_phi_r-1_asymptotics} is fairly simple. We use the fact that
$\ln(1-z)=O(|z|)$ as $z\to 0$ to obtain
$$
\ln((xq^k;q)_{\infty})=\sum\limits_{j\ge k} \ln(1-x q^j)=O(\sum\limits_{j\ge k} |x| q^j)=O(q^k), \;\;\; k\to +\infty.
$$
Thus $(xq^k;q)_{\infty}=1+O(q^k)$ as $k\to +\infty$ and
\begin{equation}\label{eq:(x;q)_k_asymptotics}
(x;q)_k=\frac{(x;q)_{\infty}}{(xq^k;q)_{\infty}}=(x;q)_{\infty}+O(q^k).
\end{equation}
Let us denote the right-hand side in \eqref{eq:r_phi_r-1_asymptotics} by $C$. From
\eqref{eq:(x;q)_k_asymptotics} we see that
$$
\frac{({\mathbf v};q)_{k}}{({\mathbf w};q)_{k}(q;q)_{k}}=C+O(q^k)
$$
and the above result implies
\begin{align*}
{_r\phi_{r-1}}\left(\!\begin{array}{c} {\mathbf v} \\ {\mathbf w}  \end{array}\!\vline\,q,z\right)
&=\sum_{k=0}^{\infty}\frac{({\mathbf v};q)_{k}}{({\mathbf w};q)_{k}(q;q)_{k}}z^k\\
&=\sum_{k=0}^{\infty}Cz^k+
\sum_{k=0}^{\infty}\Bigg[\frac{({\mathbf v};q)_{k}}{({\mathbf w};q)_{k}(q;q)_{k}}-C\Bigg] z^k\\
&=\frac{C}{1-z}+\Big\{{\textrm{a function analytic in the disk}} \; |z|<1/q \Big\}.
\end{align*}
The desired asymptotic formula \eqref{eq:r_phi_r-1_asymptotics} follows from the above result.

To derive \eqref{eq:sum-beta1} one needs to multiply both sides of
\eqref{eq:MainIdentityNew} by $(1-z)$, take a limit as $z\to 1-$ and simplify the resulting equation using identity $(x;q)_{\infty}=(x;q)_k (xq^k;q)_{\infty}$.
\end{proof}

Our next goal is an identity for terminating $q$-series.
Denote
\begin{equation}\label{eq:ABM2N2}
A=\sum a_j,~~B=\sum b_j,~~M_2=\sum m_i^2,~~N_2=\sum n_i^2.
\end{equation}
We will need an alternative asymptotic formula for $f_k(z)$. In view of
$$
(a;q)_{n}=(-a)^nq^{\binom{n}{2}}\left(\frac{q^{1-n}}{a};q\right)_{n},
$$
we have
\begin{multline*}
f_k(z)=-\frac{(zq^{1-\b};q)_{k+\m}}{(zq^{-\a};q)_{k+\n+1}}
\\
=-\frac{\left(z^{-1}q^{\b-k-\m};q\right)_{k+\m}}{\left(z^{-1}q^{\a-k-\n};\right)_{k+\n+1}}
(-z)^{M-N-r}(q^{1-\b})^{k+\m}(q^{-\a})^{-k-\n-1}q^{\binom{k+\m}{2}}q^{-\binom{k+\n+1}{2}}
\\
=-(-z)^{M-N-r}q^{k(r+A-B)+A+M+(\a,\n)-(\b,\m)}q^{k(M-N-r)+(M_2-N_2-M-N)/2}
\frac{\left(z^{-1}q^{\b-k-\m};q\right)_{k+\m}}{\left(z^{-1}q^{\a-k-\n};\right)_{k+\n+1}}
\\=(-1)^{M-N-r+1}q^{k(A-B+M-N)}q^{A+(\a,\n)-(\b,\m)+(M_2-N_2+M-N)/2}
z^{M-N-r} \hat{f}_k(z).
\end{multline*}
Then, by the standard Taylor series  $\log(1-x)=-\sum_{s\ge1}x^s/s$, we compute
\begin{multline*}
\log\hat{f}_k(z)=\sum_{i=1}^{r}\Bigg[\sum_{j=0}^{k+m_i-1}\log(1-z^{-1}q^{b_i-m_i-k+j})-\sum_{j=0}^{k+n_i}\log(1-z^{-1}q^{a_i-n_i-k+j})\Bigg]
\\
=\sum_{i=1}^{r}\Bigg[\sum_{j=0}^{k+n_i}\sum_{s\ge1}\frac{1}{s}z^{-s}q^{(a_i-n_i-k+j)s}-\sum_{j=0}^{k+m_i-1}\sum_{s\ge1}\frac{1}{s}z^{-s}q^{(b_i-m_i-k+j)s}\Bigg]
=\sum_{s\ge1}z^{-s}Q_s,
\end{multline*}
where
\begin{subequations}\label{eq:Qq}
\begin{equation}\label{eq:BigQ}
Q_s=Q_s(\a,\b,\m,\n)=\frac{1}{s}\sum_{i=1}^{r}\Bigg[\sum_{j=0}^{k+n_i}q^{(a_i-n_i-k+j)s}-\sum_{j=0}^{k+m_i-1}q^{(b_i-m_i-k+j)s}\Bigg].
\end{equation}
Then
$$
\hat{f}_k(z)=\sum\limits_{s=0}^{\infty}\frac{q_s}{z^s}
$$
with
\begin{equation}\label{eq:smnallq}
q_0=1,~~~~q_s=\frac{1}{s}\sum_{t=1}^{s}tQ_{t}q_{s-t}.
\end{equation}
\end{subequations}
This implies that the coefficient $C_{-1}(k)$ in the expansion
$$
f_{k}(z)=\sum\limits_{j=-\infty}^{M-N-r}C_{j}(k)z^{j}.
$$
is given by $p=M-N-r+1$
$$
C_{-1}(k)=(-1)^{M-N-r+1}q^{k(A-B+M-N)}q^{A+(\a,\n)-(\b,\m)+(M_2-N_2+M-N)/2}q_p(\a,\b,\m,\n)
$$

In a similar fashion in the neighbourhood of $z=0$ we calculate
$$
\log(-f_k)=\sum_{i=1}^{r}\Bigg[\sum_{j=0}^{k+m_i-1}\log(1-zq^{1-b_i+j})-\sum_{j=0}^{k+n_i}\log(1-zq^{-a_i+j})\Bigg]=\sum_{s=1}^{\infty}P_{s}z^{s},
$$
where
\begin{subequations}\label{eq:Pp}
\begin{equation}\label{eq:bigp}
P_s=P_s(\a,\b,\m,\n)=\frac{1}{s}\sum_{i=1}^{r}\Bigg[\sum_{j=0}^{k+n_i}q^{(-a_i+j)s}-\sum_{j=0}^{k+m_i-1}q^{(1-b_i+j)s}\Bigg].
\end{equation}
Hence,
$$
-f_k=e^{\log(-f_k)}=\sum_{s=0}^{\infty}p_{s}z^s
$$
with
\begin{equation}\label{eq:smnallp}
p_0=1,~~~~p_s=\frac{1}{s}\sum_{t=1}^{s}tP_{t}p_{s-t}.
\end{equation}
\end{subequations}
This yields
$$
\res\limits_{z=0}{F_k(z)}=-\res\limits_{z=0}{-z^{-t}f_k(z)}=-p_{t-1}
$$
On the other hand, from it follows from the proof of Theorem~\ref{thm:main}
that for $k\ge-m_{\min}$
$$
C_{-1}(k)- \res\limits_{z=0}{F_k(z)}=\sum_{i=1}^{r}\sum_{j=0}^{k+n_i}\gamma_{i,j}^{k+n_i}=\alpha_k.
$$
Hence we arrive at
\begin{prop}
Suppose  $k\ge-m_{\min}$  and $q_{-1}=0$, $q_0=1$ and  $q_p=q_p(\a,\b,\m,\n)$ is given by \eqref{eq:Qq} for $p\ge1$.
Similarly, $p_{-1}=0$ and $p_0=1$ and $p_{t-1}=p_{t-1}(\a,\b,\m,\n)$ is given by \eqref{eq:Pp} for $t\ge1$.
Then
\begin{multline}\label{eq:alphaIdentity}
\sum_{i=1}^{r}\frac{q^{a_i(1-t)}(q^{1-\b+a_{i}};q)_{\m+k}}{(q;q)_{k+n_i}(q^{a_{i}-\a_{[i]}};q)_{\n_{[i]}+k+1}}
{_{2r}\phi_{2r-1}}\left(\!\begin{array}{l}q^{-k-n_i},q^{\b-a_i},q^{\a_{[i]}-a_{i}-\n_{[i]}-k}\\q^{\b-a_i-\m-k},q^{1-a_i+\a_{[i]}}\end{array}\!\vline\, q^{N-M+r-1+t}\right)
\\
=(-1)^{M-N-r+1}q^{k(A-B+M-N)}q^{A+\a\cdot\n-\b\cdot\m+(M_2-N_2+M-N)/2}q_p +p_{t-1},
\end{multline}
where $A,B,M_2,N_2$ are defined in \eqref{eq:ABM2N2} and $\a\cdot\n=a_1n_1+\cdots +a_rn_r$ is the scalar product.
\end{prop}

\medskip

For small values of the components of $\n$, $\m$ the right hand sides of \eqref{eq:MainIdentityNew} and \eqref{eq:sum-beta1} turn out to be rather simple.  While formula \eqref{eq:beta-final} does provide a general form of the numbers $\beta_k$, it seems to be more practical to multiply \eqref{eq:MainIdentityNew} by $(Wz;q)_{p+1}(z;q)_{(t)_{+}}$ and compute the power series coefficients of the function on the left hand side at powers of $z$ from $-n_{\max}$ to $p+(t)_{+}-m_{\min}$.  Below we give some examples of the results obtained by this procedure for $r=3$. We present the complete right hand side of \eqref{eq:MainIdentityNew} and \eqref{eq:sum-beta1} for given $\n$, $\m$ and $t$.

\textbf{Example 1.}
Set $\n = (1, 1, 2)$, $\m = (1, 2, 2)$, $t=0$. Then the right-hand side of \eqref{eq:MainIdentityNew} takes
the form $q^{a_3}/[(1-q^{a_3-b_1})z^2]$. %checked...
\bigskip

\textbf{Example 2.}
Set $\n = (0, 2, 3)$, $\m = (2, 2, 2)$, $t=0$. Then the right-hand side of \eqref{eq:MainIdentityNew} takes
the form
$$
\frac{ (1-q^{a_3-a_1-1})(1-q^{a_3-a_1-2})q^{a_3+b_1+b_2+b_3}}{(q^{b_1}-q^{a_3})(q^{b_2}-q^{a_3})(q^{b_3}-q^{a_3})z^3}.
$$
%checked...
\bigskip

\textbf{Example 3.}
Set $\n = (0, 0, 1)$, $\m = (0, 1, 1)$, $t=0$. Then the right-hand side of \eqref{eq:MainIdentityNew} takes
the form
$$
\frac{q^{a_3}}{(1-q^{a_3-b_1})z}.
%checked...
$$

\textbf{Example 4.}
Set $\n = (0, 0, 1)$, $\m = (0, 1, 1)$, $t=1$. Then the right-hand side of \eqref{eq:MainIdentityNew} takes
the form
$$
\frac{1}{(1-q^{a_3-b_1})z(1-z)}+\frac{1}{(1-q^{b_1-a_3})(1-z)}.
$$

\textbf{Example 5.}
Set $\n = (0, 0, 1)$, $\m = (0, 1, 1)$, $t=-1$. Then the right-hand side of \eqref{eq:MainIdentityNew} takes
the form
$$
\frac{q^{2a_3}}{(1-q^{a_3-b_1})z(1-Wz)}+\dfrac{q^{1+a_1+a_2+2a_3+b_1-b_2-b_3}}{(q^{a_3}-q^{b_1})(1-Wz)}
$$

\section{The confluent case}

In the following theorem we give an analogue of Theorem~\ref{thm:main}  for the function ${}_s\phi_{r-1}$ with $s<{r}$.  The result is given in terms of the mixture of ${}_s\phi_{r-1}$ and ${}_s\hat{\phi}_{r-1}$ functions.  Using formula \eqref{eq:phitildephi} it can be rewritten in terms of the standard  ${}_s\phi_{r-1}$  functions only.  We retain the notation
$$
m_{\min}=\min\limits_{1\le i \le s}(m_i),~~n_{\max}=\max\limits_{1\le i \le r}(n_i),~~M=\sum\limits_{i=1}^s m_i, ~~N=\sum\limits_{i=1}^{r}n_i
$$
from Theorem~\ref{thm:main}.

\begin{theorem}\label{thm:confluent}
Assume that $q$ is a complex number satisfying $|q|<1$,  $s$ and $r$ are integers such that $0\leq {s}<r$, $\a\in\mathbb{C}^r$ satisfies $a_i-a_j\notin\Z$ for $1\le i < j\le r$ and  $\b\in\mathbb{C}^{s}$ is  arbitrary; assume further that $\m\in\Z^s$, $\n\in\Z^r$ and $t\in \Z$. Then, for the numbers $\beta_k$ given by
\begin{multline}\label{eq:beta-confluent}
\beta_{k}=\sum\limits_{j=\max(-n_{\max},k-(t)_{+})}^{k}\qbinom{(t)_{+}}{k-j}_q q^{(k-j)(k-j-1)/2}(-1)^{k-j}
\\
\times\sum_{i=1}^{r}\frac{q^{a_i(1-t)}(q^{1-\b+a_{i}};q)_{\m+j}}{(q;q)_{j+n_i}(q^{a_{i}-\a_{[i]}};q)_{\n_{[i]}+j+1}}
{_{r+s}\phi_{r+s-1}}\left(\!\begin{array}{l}q^{-j-n_i},q^{\b-a_i},q^{\a_{[i]}-a_{i}-\n_{[i]}-j}\\q^{\b-a_i-\m-j},q^{1-a_i+\a_{[i]}}\end{array}\!\vline\, q^{N-M+r-1+t+j(r-s)}\right)
\end{multline}
the following identity holds
\begin{multline}\label{eq:_t_r}
\sum_{i=1}^{r} q^{a_i(1-t)}\frac{(q^{1-\b+a_i};q)_{\m-n_i}z^{-n_i}}{(q^{a_i-\a_{[i]}};q)_{\n_{[i]}-n_i+1}}{_s\phi_{r-1}}\left(\!\begin{array}{l}q^{\b-a_i}\\q^{1+\a_{[i]}-a_i}\end{array}\!\vline\, W q^{(s-r)a_i}z\right)
\\
\times
{}_s\hat{\phi}_{r-1}\left(\!\begin{array}{l}q^{1-\b+a_i+\m-n_i}\\q^{1-\a_{[i]}+a_i+\n_{[i]}-n_i}\end{array}\!\vline\,z\right)
=\frac{1}{(z;q)_{(t)_{+}}}\sum_{k=-n_{\max}}^{K} \beta_{k}z^k,
\end{multline}
where  $p'=[(M-N-r-t+1)/(r-s)]$,  $K=\max(-m_{\min}-1,p')+(t)_{+}$
and
$$
W=q^{t+r-1}\prod\limits_{i=1}^r q^{a_i} \prod\limits_{i=1}^s q^{-b_i}.
$$
\end{theorem}
\begin{proof}
Repeating the proof of Theorem~\ref{thm:main} we compute
\begin{multline*}
S(z)=\sum_{i=1}^{r}q^{a_i(1-t)}\frac{(q^{1-\b+a_i};q)_{\m-n_j}z^{-n_i}}{(q^{a_i-\a_{[i]}};q)_{\n_{[i]}-n_i+1}}
\\
\times{_s\phi_{r-1}}\left(\!\begin{array}{l}q^{\b-a_i}\\q^{1+\a_{[i]}-a_i}\end{array}\!\vline\, W q^{(s-r)a_i}z\right)
{_s\hat{\phi}_{r-1}}\left(\!\begin{array}{l}q^{1-\b+a_i+\m-n_i}\\q^{1-\a_{[i]}+a_i+\n_{[i]}-n_i}\end{array}\!\vline\,z\right)
\\
=\sum_{i=1}^{r}\sum_{k=0}^{\infty}z^{k-n_{i}}\sum_{j=0}^{k}
\frac{q^{a_i(1-t)}(q^{1-\b+a_i};q)_{\m-n_i}(q^{\b-a_i};q)_{j}(q^{1-\b+a_i+\m-n_i};q)_{k-j}(Wq^{(s-r)a_i})^{j}}
{(q^{a_i-\a_{[i]}};q)_{\n_{[i]}-n_i+1}(q^{1+\a_{[i]}-a_i};q)_{j}(q^{1-\a_{[i]}+a_i+\n_{[i]}-n_i};q)_{k-j}(q;q)_{j}(q;q)_{k-j}}
\\
\times\left[(-1)^{j} q^{j(j-1)/2}\right]^{r-s}
=\sum_{i=1}^{r}\sum_{k=0}^{\infty}z^{k-n_i}\sum_{j=0}^{k}\gamma_{i,j}^k =\sum_{i=1}^{r}\sum_{k_i=-n_i}^{\infty}z^{k_i}\sum_{j=0}^{k_i+n_i} \gamma_{i,j}^{k_i+n_i},
\end{multline*}
where
\begin{equation}\label{eq:gamma_ij_thm2}
\gamma_{i,j}^k=\frac{q^{a_i(1-t)}(q^{1-\b+a_i};q)_{\m-n_i}(q^{\b-a_i};q)_{j}(q^{1-\b+a_i+\m-n_i};q)_{k-j}(Wq^{(s-r)a_i})^{j}}
{(q^{a_i-\a_{[i]}};q)_{\n_{[i]}-n_i+1}(q^{1+\a_{[i]}-a_i};q)_{j}(q^{1-\a_{[i]}+a_i+\n_{[i]}-n_i};q)_{k-j}(q;q)_{j}(q;q)_{k-j}}\left[(-1)^{j} q^{j(j-1)/2}\right]^{r-s}.
\end{equation}
After a calculation similar to that in the proof of Theorem~\ref{thm:main}, we arrive at
\begin{multline*}\label{eq:gamma_compact}
\gamma_{i,j}^k=\frac{(-1)^{(s-r+1)j}q^{(s-r+1)j(j-1)/2}q^{a_i(1-t)+tj}(q^{1-\b+a_{i}-j};q)_{\m+k-n_{i}}}{(q^{a_{i}-\a_{[i]}-j};q)_{\n_{[i]}+k-n_{i}+1}(q;q)_{j}(q;q)_{k-j}}\left[(-1)^{j} q^{j(j-1)/2}\right]^{r-s}
\\
=\frac{(-1)^{j}q^{j(j-1)/2}q^{a_i(1-t)+tj}(q^{1-\b+a_{i}-j};q)_{\m+k-n_{i}}}{(q^{a_{i}-\a_{[i]}-j};q)_{\n_{[i]}+k-n_{i}+1}(q;q)_{j}(q;q)_{k-j}}
\end{multline*}
or
\begin{equation}\label{eq:gamma_compact2}
\gamma_{i,j}^{k+n_{i}}=\frac{(-1)^{j}q^{j(j-1)/2}q^{a_i(1-t)+tj}(q^{1-\b+a_{i}-j};q)_{\m+k}}{(q^{a_{i}-\a_{[i]}-j};q)_{\n_{[i]}+k+1}(q;q)_{j}(q;q)_{k+n_{i}-j}}.
\end{equation}

Again, according to \eqref{eq:qPochhammerId} we have $(q;q)_{j}=\infty$ for $j<0$ so that also $(q;q)_{k-j}=\infty$ for $j>k$ and thus formula \eqref{eq:gamma_compact2} implies that $\gamma_{i,j}^{k+n_i}=0$ for $j<0$ and $j>k+n_i$. This relations extend definition of $\gamma_{i,j}^{k+n_i}$ to arbitrary $j$, and in view of this convention, we have
$$
S(z)=\sum_{k=-n_{\max}}^{\infty}z^k\sum_{i=1}^{r}\sum_{j=0}^{k+n_i}\gamma_{i,j}^{k+n_i}.
$$
Next, note that $k+m_i\ge0$ for all $i=1,\dots,r$ if $k\ge-m_{\min}$. If $-m_{\min}\le -n_{\max}$, then $k\ge -m_{\min}$ for all terms in the above sum. Otherwise, if $-m_{\min}>-n_{\max}$ we can write
$$
S(z)=\sum_{k=-n_{\max}}^{-m_{\min}-1}\alpha_k z^k+\sum_{k=-m_{\min}}^{\infty}z^k\sum_{i=1}^{r}\sum_{j=0}^{k+n_i} \gamma_{i,j}^{k+n_i}
=\sum_{k=-n_{\max}}^{-m_{\min}-1}\alpha_kz^k +S_1(z),
$$
where
$$
\alpha_k=\sum_{i=1}^{r}\sum_{j=0}^{k+n_i}\gamma_{i,j}^{k+n_i}.
$$

Next, for each $k\in\Z$ define the functions
\begin{equation}\label{eq:f_k2}
f_k(z)=-\frac{(zq^{1-\b};q)_{k+\m}}{(zq^{-\a};q)_{k+\n+1}}~~\text{and}~~F_k(z)=z^{-t} f_k(z).
\end{equation}
These functions are rational for all $k\in\Z$.  Further, if  $k\ge-m_{\min}$ the expression
in the numerator in \eqref{eq:f_k2} is a polynomial and all poles of $f_k(z)$ come from the zeros of the denominator.  If $k+n_i+1>0$ for all $i=1,\ldots,r$, then the poles of $f_k(z)$ are at the points:
$$
(zq^{-\a};q)_{k+\n+1}=0~~\Leftrightarrow~~z=q^{a_{i}-j},~i=1,\ldots,r~\text{and}~j=0,\ldots,k+n_{i}.
$$
Since the poles are simple, we have after some rearrangement:
$$
\res\limits_{z=q^{a_{i}-j}}F_k(z)
=\frac{(-1)^{j}q^{j(j-1)/2}q^{a_i(1-t)+tj}(q^{1-\b+a_{i}-j};q)_{\m+k}}
{(q^{a_{i}-\a_{[i]}-j};q)_{\n_{[i]}+k+1}(q;q)_{j}(q;q)_{k+n_{i}-j}}
=\gamma_{i,j}^{k+n_{i}}.
$$
It is clear that
$$
F_{k}(z)=\sum\limits_{j=-\infty}^{M-N-r-t-(r-s)k}C_{j}(k)z^{j}.
$$
It implies that $C_{-1}(k)=0$ for $M-N-r-t-(r-s)k<-1$, i.e. for $(r-s)k>M-N-r-t+1$.
Using the $q$-binomial theorem in the form
$$
\frac{1}{(z;q)_k}=\sum\limits_{j\ge 0} \frac{(q^k;q)_j}{(q;q)_j} z^j
$$
and the results of Lemma \ref{lemma1} we see that
$$
f_k(z)=\sum\limits_{j\ge 0} Q_j(q^k) z^j
$$
for some polynomial $Q_j$ of degree $j$ which does not depend on $k$. Thus the residue of $F_k(z)$ at $z=0$ is equal to $Q_{t-1}(q^k)$ if $t\ge 1$ and we obtain
$$
\alpha_k=\sum_{i=1}^{r}\sum_{j=0}^{k+n_i}\gamma_{i,j}^{k+n_i}=C_{-1}(k)+Q_{t-1}(q^k) {\mathbf 1}_{\{t\ge 1\}}
$$

Then if $t\le 0$ we have
$$
S_1(z) = \sum_{k=-m_{\min}}^{\infty} \alpha_k z^k= \left\{\begin{array}{ll}\sum_{k=-m_{\min}}^{(r-s)k\le M-N-r-t+1} C_{-1}(k) z^k, & -p'\le m_{\min}
\\
\\
0, & -p'>m_{\min}\end{array}\right.
$$
and if $t\ge 1$ and $-p' \le m_{\min}$ we have
\begin{align*}
S_1(z)=\sum_{k=-m_{\min}}^{\infty}\alpha_k z^k
=\sum_{k=-m_{\min}}^{(r-s)k\le M-N-r-t+1}C_{-1}(k)z^k
+\sum\limits_{j=0}^{t-1}\frac{\beta_{t-1,j}}{1-zq^t}
\end{align*}
and if $t\ge 1$ and $-p' > m_{\min}$ we have
\begin{align*}
S_1(z)=\sum_{k=-m_{\min}}^{\infty} \alpha_k z^k=  \sum\limits_{j=0}^{t-1} \frac{\beta_{t-1,j}}{1-zq^t}.
\end{align*}
Finally,
$$
S(z)=\sum_{-n_{\max}}^{-m_{\min}-1}\alpha_{k}z^k+S_1(z)=\frac{1}{(z;q)_{(t)_{+}}}
\sum_{j=-n_{\max}}^{\max(-m_{\min}-1,p')+(t)_{+}} \beta_j z^j.
$$

Next we compute the numbers $\beta_k$.  Application of the  identities \eqref{eq:qpochhammertricks} to \eqref{eq:gamma_compact2} after some simplifications leads to
\begin{equation*}
\gamma_{i,j}^{k+n_{i}}=\frac{q^{a_i(1-t)}q^{(N-M+r-1+t+k(r-s))j}(q^{1-\b+a_{i}};q)_{\m+k}(q^{-k-n_i};q)_{j}(q^{\b-a_{i}};q)_{j}(q^{\a_{[i]}-a_{i}-\n_{[i]}-k};q)_{j}}
{(q;q)_{k+n_i}(q^{a_{i}-\a_{[i]}};q)_{\n_{[i]}+k+1}(q^{\b-a_i-\m-k};q)_{j}(q^{1-a_i+\a_{[i]}};q)_{j}(q;q)_{j}},
\end{equation*}
so that
\begin{multline*}
\alpha_k=\sum_{i=1}^{r}\sum_{j=0}^{k+n_i}\gamma_{i,j}^{k+n_i}
\\
=\sum_{i=1}^{r}\frac{q^{a_i(1-t)}(q^{1-\b+a_{i}};q)_{\m+k}}{(q;q)_{k+n_i}(q^{a_{i}-\a_{[i]}};q)_{\n_{[i]}+k+1}}
{_{r+s}\phi_{r+s-1}}\left(\!\begin{array}{l}q^{-k-n_i},q^{\b-a_i},q^{\a_{[i]}-a_{i}-\n_{[i]}-k}\\q^{\b-a_i-\m-k},q^{1-a_i+\a_{[i]}}\end{array}\!\vline\, q^{N-M+r-1+t+k(r-s)}\right),
\end{multline*}
where each term with $k+n_i<0$ is assumed to be equal to zero.  Next, according to the Gauss expansion
$$
(z;q)_{(t)_{+}}=\sum\limits_{j=0}^{(t)_{+}}\qbinom{(t)_{+}}{j}_q q^{j(j-1)/2}(-z)^j,
$$
where $\qbinom{(t)_{+}}{j}_q$ are the $q$-binomial coefficients.  Hence, multiplying both sides by $z^{n_{\max}}$, changing $k+n_{\max}\to{k}$ and writing $\hat{\beta}_k=\beta_{k-n_{\max}}$,
$\hat{\alpha}_k=\alpha_{k-n_{\max}}$, we get
$$
\sum\limits_{j=0}^{(t)_{+}}\qbinom{(t)_{+}}{j}_q q^{j(j-1)/2}(-z)^j
\sum\limits_{k=0}^{\infty}\hat{\alpha}_{k}z^k
=\sum\limits_{s=0}^{\infty}z^{s}\sum\limits_{j+k=s}\qbinom{(t)_{+}}{j}_q
q^{j(j-1)/2}(-1)^j\hat{\alpha}_{k}
=\sum\limits_{k=0}^{K+n_{max}}\hat{\beta}_{k}z^{k}.
$$
In view of $\qbinom{(t)_{+}}{j}_q=0$ for $j>(t)_{+}$ this implies that
$$
\beta_{s-n_{\max}}=\hat{\beta}_s=\sum\limits_{j=0}^{\min(s,(t)_{+})}\qbinom{(t)_{+}}{j}_q q^{j(j-1)/2}(-1)^j\hat{\alpha}_{s-j}
=\sum\limits_{j=0}^{\min(s,(t)_{+})}\qbinom{(t)_{+}}{j}_q q^{j(j-1)/2}(-1)^j\alpha_{s-j-n_{\max}}
$$
for $s=0,\ldots,K+n_{\max}$. Returning to $k=s-n_{\max}$ we obtain by changing the index of summation $j\to{k-j}$:
\begin{multline*}
\beta_{k}=\sum\limits_{j=0}^{\min(k+n_{\max},(t)_{+})}\qbinom{(t)_{+}}{j}_q q^{j(j-1)/2}(-1)^j\alpha_{k-j}
\\
=\sum\limits_{j=\max(-n_{\max},k-(t)_{+})}^{k}\qbinom{(t)_{+}}{k-j}_q q^{(k-j)(k-j-1)/2}(-1)^{k-j}\alpha_{j}
\end{multline*}
for $k=-n_{\max},\ldots,K$. Substituting the expression for $\alpha_{j}$ derived above,  we finally arrive at  \eqref{eq:beta-confluent}.
\end{proof}

In the same way as we proved Proposition \ref{prop:sum-beta1} above, one can obtain the following result (all details are left to the reader).
\begin{prop}\label{prop:sum-beta2}
Assume that all conditions of Theorem \ref{thm:confluent} are satisfied. Then
\begin{equation}\label{eq:sum-beta2}
\sum_{i=1}^{r} q^{a_i(1-t)}\frac{(q^{1-\b+a_i};q)_{\infty}}{(q^{a_i-\a_{[i]}};q)_{\infty}}{_s\phi_{r-1}}\left(\!\begin{array}{l}q^{\b-a_i}\\q^{1+\a_{[i]}-a_i}\end{array}\!\vline\, W q^{(s-r)a_i}\right)
=
(q^t;q)_{\infty}\sum\limits_{k=-n_{max}}^{K}\beta_k.
\end{equation}
Note that $(q^t;q)_{\infty}=0$ for $t\le 0$ , so that the right-hand side vanishes for such values of $t$.
\end{prop}

\bigskip

\paragraph{{\bf Acknowledgements.}}  We thank anonymous referees for a number of useful suggestions that lead to a substantial improvement of this paper. Research of A.K. was supported by
 the Natural Sciences and Engineering Research Council of Canada.

\bigskip

\bibliographystyle{amsplain}

\end{document}